\documentclass[letterpaper, 10 pt, conference]{ieeeconf}  
\IEEEoverridecommandlockouts                              
\usepackage{subfigure}
\usepackage{amsmath}
\usepackage{amsfonts}
\usepackage{amssymb}
\usepackage{color}
\usepackage{fixltx2e}
\usepackage{algorithm}
\usepackage{algorithmic}
\usepackage{graphicx}
\usepackage[compress]{cite}

\usepackage{bbm}

\newtheorem{theorem}{Theorem}[section]
\newtheorem{lemma}[theorem]{Lemma}

\newtheorem{assumption}[theorem]{Assumption}

\newtheorem{definition}[theorem]{Definition}
\newtheorem{remark}[theorem]{Remark}

\newcommand{\crb}{\color{black}}
\newcommand{\crr}{\color{black}}

\newcommand{\crg}{\color{black}}

\title{\LARGE \bf
Distributed Convex Optimization in Networks of Agents with Single-integrator Dynamics}

\author{ Amir Adibzadeh, Mohsen Zamani,  Amir A. Suratgar, and Mohammad B. Menhaj
\thanks{M. Zamani is with the School of Electrical Engineering and Computer Science, The University of Newcastle, Callaghan, NSW 2308, Australia. e-mails:{\tt\small  mohsen.zamani@newcastle.edu.au.}}
\thanks{A. Adibzadeh, A. A. Suratgar and M. B. Menhaj  are with the Electrical Engineering
	Department, Amirkabir University of Technology, Tehran 15914, Iran. e-mail:{\tt\small \{amir.adibzadeh,a-suratgar,menhaj\}@aut.ac.ir.}}}
\begin{document}

\maketitle
\thispagestyle{empty}
\pagestyle{empty}

{\crg
	\begin{abstract}
		In this paper, distributed convex optimization problem
		over undirected dynamical networks is studied. Here, networked agents
		with single-integrator dynamics are supposed to rendezvous at a point
		that is the solution of a global convex optimization problem with some local
		inequality constraints. To this end, all agents shall cooperate with
		their neighbors to seek the optimum point of the network's global  objective function. A distributed optimization algorithm based on the interior-point
		method is proposed, which combines an optimization algorithm with
		a nonlinear consensus protocol to find the optimum value of the global objective function. {\crb We tackle
			this problem by addressing its subproblems, namely a consensus problem
			and a convex optimization problem. Firstly, we propose a saturation protocol for the consensus subproblem. Then to solve the distributed optimization part, we implement a centralized control law, which yields the optimum value of the global objective function,  in a distributed fashion with the help of a distributed estimator.  Convergence analysis for the proposed protocol } based
		on the Lyapunov stability theory for time-varying nonlinear systems is
		included. A simulation example is given at the end to illustrate the
		effectiveness of the proposed algorithm.
	\end{abstract}
}
\section{INTRODUCTION}

{\crb
	In recent years,  developing distributed  paradigms  for solving  optimization problems among interconnected agents has attracted attention of researchers. 
	We briefly review some of the existing  works in this area.}

The authors of \cite{raffard2004distributed} used {\crr the dual decomposition scheme, which maintains a small duality gap,} to solve optimization problems
in a network of dynamical nonlinear agents. { The reference \cite{nedic2009distributed}  exploited a {\crr subgradient-based} distributed method to find the approximation of an optimal point associated with a collective convex function over a network
	of interconnected agents. The paper \cite{lu2012zero} proposed
	a zero-gradient-sum continuous-time algorithm  to drive the states of a weight-balanced directed network to {\crr the optimal point of a} global objective function along an invariant zero-gradient-sum manifold. The references \cite{yuan2011distributed,yi2015distributed,kia2015distributed} exploited   { \crr the dual decomposition method} to deal with distributed optimization problems with inequality
	and equality constraints { \crr over networks.} In these studies, to find the saddle point of {\crr the Lagrangian corresponding with the original  optimization problem},
a distributed {\crr continuous-time gradient-based dynamics was developed for primal and dual
decision variables associated with each agent.}   The
paper \cite{qiu2016distributed} utilized  a consensus protocol  to tackle a distributed optimization problem for networks, { in which   agents   share a  common convex constraint. A comparison between the dual decomposition-based method and the consensus-based method
	for distributed optimization in networked systems  was studied in \cite{droge2014continuous}. {\crr The authors} combined these two methods and developed a continuous-time
	proportional-integral distributed gradient-based technique. In order to solve constrained optimization problems and attain  the saddle
	point of the corresponding Lagrangian function, proper dynamics associated with primal and dual variables were developed in \cite{feijer2010stability} that yield the optimal  solution. In the above mentioned works, the complexity associated with the implementation of proposed algorithms increases as the number of agents or that of constraints corresponding with them becomes greater.
}

In this paper, we consider the constrained distributed optimal problem for single-integrator networks, where each agent has a convex objective function and personalized  inequality constraints.
To solve the problem, we divide it into {\crr a consensus problem} and a
distributed optimization one. To deal with the former problem, we
utilize a continuous  consensus protocol based
on local information sharing. Then, we exploit a distributed optimization algorithm based on the interior-point method to solve
the convex optimization problem.\textcolor{black}{{} {\crb In the proposed
		algorithm, no Lagrangian variables associated with the consensus constraint, {\crr which is  needed so that all agents attain the same optimum solution},
		and  {\crr the} local inequality constraints are required.
		This reduces the complexity of the proposed solution and its implementation.  Moreover, the proposed algorithm can handle the limitations associated with actuator saturation  that commonly occurs in practice.  }}

{\crb
	 This paper is structured as follows. The next section reviews some background materials required in this paper. The problem formulation and main results are introduced in  Section \ref{sec:main}. Finally, Section \ref{sec:conclude} concludes the paper.
}

\section{Preliminaries and Notations}
{\crb In this section, we briefly review some background materials required in this paper.}

\subsection{Graph Theory}

In graph theory, $\mathcal{{G}=\left\{ \text{\ensuremath{\mathcal{N\text{,}E\text{,}A}}}\right\} }$
denotes an undirected network, where $\mathcal{{N}}=\{1,\cdots,N\}$
is the set of nodes. An edge between node $i$ and node $j$ is denoted by
the pair $(i,j)\in\mathcal{{E}}$ that indicates mutual communication
between two nodes $i$ and $j$. The set  $\mathcal{{E}}\subseteq\mathcal{N\times N}$
represents the set of edges, and $\mathcal{{A}=}[a_{ij}]_{N\times N}$
is the adjacency matrix. $\mathcal{{A}}$ is symmetric {\crb and}  $a_{ij}=1$
when $(i,j)\in\mathcal{{E}}$, and $a_{ij}=0$ indicates $(i,j)\notin\mathcal{{E}}$.
It is assumed that there is no repeated edge and no self-loop, i.e.
$a_{ii}=0$.  The set of neighbors of node $i$
is denoted by $\mathcal{{N_{\text{i}}}}=\{j\in\mathcal{{V}}:(i,j)\in\mathcal{{E}}\}$.
Assume an arbitrary orientation for each edge in $\mathcal{{G}}$,
then $D=[d_{ik}]\in\mathbb{{R}}^{N\times|\mathcal{{E}}|}$ is the
incidence matrix associated with $\mathcal{{G}}$, in which $d_{ik}=-1$
if the edge $(i,j)$ leaves node $i$, $d_{ik}=1$ if it enters the
node, and $d_{ik}=0$ otherwise. The Laplacian matrix $L=[l_{ij}]\in\mathcal{\mathbb{{R}}}^{N\times N}$
associated with the graph $\mathcal{{G}}$ is defined as $l_{ii}=\sum_{j=1,j\neq i}^{N}a_{ij}$
and $l_{ij}$$=-a_{ij}$ for $i\neq j$ . \textcolor{black}{Note that
	$L=DD^{\top}$. The Laplacian matrix $L$ is semi-positive {\crb definite, and  if} $\mathbf{1\in\mathbb{R}}^{n}$ denotes a vector of which entities
	are all 1, then, $L\mathbf{1}=\mathbf{0}$ and $\mathbf{1}^{\top}L=\mathbf{0}$. }{\crb The Laplcian matrix $L$ has one zero eigenvalue
	if the graph $\mathcal{{G}}$ is connected. All eigenvalues of $L$ are
	non-negative. We define consensus error in a network by $\bar{e}_{x}=\Pi\bar{x}$
	where $\Pi=I_{N}-\frac{1}{N}\mathbf{1}_{N}\mathbf{1}_{N}^{\top}$,
	and $\bar{x}$ denotes the aggregate state of the network as $\bar{x}=\left[x_{1}\ldots x_{N}\right]^{\top}$.
	Note that $\mathbf{1}^{\top}\Pi=\mathbf{0}$ and $\Pi\mathbf{1=0}$.\textcolor{black}{{} }}

\subsection{Stability of Perturbed Systems}

\textcolor{black}{Consider the nominal system
	\begin{equation}
	\dot{x}=f(x,t).\label{eq:Perl2}
	\end{equation}
	where $f:\mathcal{{D}}\times[0,\infty)\rightarrow\mathbb{R}^{N}$
	is piecewise continuous in $t$ and locally Lipschitz in $x$ on $\mathcal{{D}}\times[0,\infty)$,
	and $\mathcal{{D}}\subset\mathbb{R}^{N}$ is a domain that contains
	the origin $x=0$. Suppose that the system \eqref{eq:Perl2} is perturbed
	by the term $d(x,t)$, where $d:\mathcal{D}\times[0,\infty)\rightarrow\mathbb{R}^{N}$
	denotes perturbation and is called unvanished perturbation if $d(0,t)\neq0$.
	Then, the perturbed system corresponding to \eqref{eq:Perl2} is given
	by
	\begin{equation}
	\dot{x}=f(x,t)+d(x,t).\label{eq:Prel1}
	\end{equation}
}
\begin{lemma}
	\label{lem:stability}\cite[Theorem 5.1]{khalil1996nonlinear} Let
	$V:\mathcal{D}\times[0,\infty)\rightarrow\mathbb{R}^{N}$ be a continuously
	differentiable function such that
	\begin{eqnarray*}
		W_{1}(x) & \leq & V(x,t)\leq W_{2}(x)\\
		\frac{\partial V(x,t)}{\partial t}+\frac{\partial V}{\partial x}f(x,t) & \leq & -W_{3}(x),\,\,\,\forall\left\Vert x\right\Vert \geq\mu>0,
	\end{eqnarray*}
	$\forall t\geq0$, $\forall x\in\mathcal{D}$, where $W_{1}(x)$,
	$W_{2}(x)$, and $W_{3}(x)$ are continuous positive definite functions
	on $\mathcal{D}$. Take $r>0$ such that $B_{r}\subset\mathcal{D}$.
	Suppose that $\mu$ is small enough such that
	\[
	\underset{\left\Vert x\right\Vert \leq\mu}{\text{{max}}}W_{2}(x)<\underset{\left\Vert x\right\Vert =r}{\text{{min}}}W_{1}(x)
	\]
	Consider $\eta=\text{{max}}_{\left\Vert x\right\Vert \leq\mu}W_{2}(x)$
	and take $\rho$ such that $\eta<\rho<\text{{min}}_{\left\Vert x\right\Vert =r}W_{1}(x)$.
	Then, there exists a finite time $t_{1}$ (dependent on $x(t_{0})$
	and $\mu$) such that $\forall\,x(t_{0})\in\{x\in B_{r}|W_{2}(x)\leq\rho\}$,
	the solutions of $\dot{x}=f(x,t)$ satisfy $x(t)\in\{x\in B_{r}|W_{1}(x)\leq\rho\},\forall t\geq t_{1}$.
	Moreover, if $\mathcal{D}=\mathbb{{R}}^{N}$ and $W_{1}(x)$ is radially
	unboudned, then this result holds for any initial state and any $\mu$.
\end{lemma}

\subsection{Notations}

Throughout this paper, $\left\Vert \cdot\right\Vert _{1}$ and $\left\Vert \cdot\right\Vert $
denote 1-norm and 2-norm operators, respectively. $\text{\ensuremath{\mathbb{{R}}}}$
represents the real numbers set and $\mathbb{{R}}^{+}$ implies the
positive real numbers subset. $\mathbb{{R}}^{N}$ includes all vectors
with $N$ real elements. The term  $\mathbb{{R}}^{N\times N}$ represents the
set of all $N\times N$ matrices with real entries. Furthermore,  $[\mathcal{{M}}_{ij}]_{N\times N}$
represents an $N\times N$ matrix with entries $\mathcal{{M}}_{ij}$,
where the index $i$ stands for the $i$-th row and $j$ refers to
$j$-th column.

\section{Problem Statement and Main Results}\label{sec:main}

Consider the single-integrator dynamics
\begin{equation}
\ensuremath{\begin{array}{l}
	{\dot{x}}(t)={u}(t),\end{array}}\label{eq:1SingIntDym-1}
\end{equation}
where $u\in\mathbb{{R}}$ and $x\in\mathbb{{R}}$ denote the  state and
control input, respectively. Assume an objective function, say $Q(x,t):\mathbb{{R}\times\mathbb{{R}\rightarrow\mathbb{{R}}}}$,
that is twice continuously differentiable and strictly convex in $x$.
\begin{lemma}
	The following control input will make the dynamics \eqref{eq:1SingIntDym-1}
	converge to the minimizer of the time-varying convex objective function
	$Q(x,t)$.
\end{lemma}
\begin{equation}
u(t)=-\left(\frac{\partial^{2}Q(x,t)}{\partial x^{2}}\right)^{-1}\left(\frac{\partial Q(x,t)}{\partial x}+\frac{\partial^{2}Q(x,t)}{\partial x\partial t}\right)\label{eq:1CentLaw-1}
\end{equation}

\begin{proof}
	Choose a Lyapunov function as $V(x,t)=\frac{1}{2}\left(\frac{\partial Q(x,t)}{\partial x}\right)^{2}$and
	take its time derivative along the {\crb trajectories of the } dynamics \eqref{eq:1SingIntDym-1}.
	Then, we have
	\[
	\dot{V}(x,t)=-\left(\frac{\partial Q(x,t)}{\partial x}\right)\left(\frac{\partial^{2}Q(x,t)}{\partial x^{2}}\dot{x}+\frac{\partial^{2}Q(x,t)}{\partial x\partial t}\right).
	\]
	By substituting $\dot{x}$ in the above relation from \eqref{eq:1CentLaw-1},
	the following is obtained
	\begin{eqnarray}
	\dot{V}(x,t) & = & -\left(\frac{\partial Q(x,t)}{\partial x}\right)^{2}\leq0\label{eq:Lemma1_1}
	\end{eqnarray}
	From the above inequality, it follows that ${\displaystyle \frac{\partial Q(x,t)}{\partial x}}$
	remains bounded in $\mathbb{{R}^{\text{n}}\bigcup}\{\infty\}$, i.e.
	it belongs to $\mathcal{{L}^{\infty}}$ space. With integrating from
	both sides of equality \eqref{eq:Lemma1_1}, in the view of passivity
	of $V(x,t)$, we have
	\begin{gather}
	\int_{0}^{R}\left(\frac{\partial Q(x,t)}{\partial x}\right)^{2}dt=-\int_{0}^{R}\dot{V}(x,t)dt\nonumber \\
	={\crr -V(x(R),R)+V(x(0),0)\leq V(x(0),0)}.\label{eq:2The3-1-1}
	\end{gather}
	So, ${\displaystyle \frac{\partial Q(x,t)}{\partial x}}\in\mathcal{{L}}^{2}$.
	Now, we invoke Barbalat's Lemma \cite{tao1997simple} and obtain
	that ${\displaystyle \frac{\partial Q(x,t)}{\partial x}}$ asymptotically
	converges to zero as $t\rightarrow\infty$. Thereby, the optimality
	condition is certified, i.e. ${\displaystyle \frac{\partial Q(x,t)}{\partial x}=0}$.
\end{proof}

Now, consider the following autonomous agents under {\crb the topology
	$\mathcal{{G}}$.} Each agent is described by the continuous-time
single-integrator dynamics:

\begin{equation}
\ensuremath{{\dot{x}_{i}}(t)={u_{i}}(t),}\quad{i}\in\mathcal{{N}}\label{eq:AgentsDyn}
\end{equation}
where $x_{i}(t)$ and $\,u_{i}(t)\in\mathbb{R}$ represent the position
and the control input to agent $i$, respectively. For the sake of
notational brevity, we will use $x_{i}$ and $u_{i}$ in the rest
of this paper. We suppose that agent $i,\,\forall{i}\in\mathcal{{N}}$,
can share its state's information with agents {\crr within its neighborhood}
set, i.e. $\mathcal{{N}}_{i}$, according to the communication graph
$\mathcal{{G}}.$

The agents are supposed to rendezvous at a point that shall minimize
the aggregate convex function $\sum_{i=1}^{N}f_{i}(x)$ with regards
to individual convex inequalities $g_{i}(x)\leq0$, $i=1,\ldots,N.$
This problem can be described by

\begin{equation}
\begin{split}&\underset{x}{min}\,F(x)=\sum_{i=1}^{N}f_{i}(x), i\in\mathcal{N}\\
&\text{subject to}\,g_{i}(x)\leq0,i\in\mathcal{N}
\end{split}
\label{eq:min Prob}
\end{equation}
in which $f_{i}\left(\cdot\right):\mathbb{R\rightarrow\mathbb{R}}$
is the local objective function associated with node $i$ and $g_{i}\left(\cdot\right):\mathbb{R\rightarrow\mathbb{R}}$
\textcolor{black}{represents a constraint on the optimal position
	imposed by $i$-th agent. }It is supposed that each agent only has
the information of its own local objective function and states of
those agents within the {\crr set of its neighbors.}

\textcolor{black}{}We express the above explained problem as the
following convex optimization problem,
\begin{equation}
\begin{split}&\underset{\underset{i=1,\ldots,N}{x_{i}}}{min}\sum_{i=1}^{N}f_{i}(x_{i}),\\
&\text{subject to}\,\begin{cases}
g_{i}(x_{i})\leq0,\,i\in\mathcal{N}\\
x_{i}=x_{j}, \,\forall i,j\in\mathcal{N}.
\end{cases}
\end{split}
\label{eq:Dis Min Prob}
\end{equation}
In the minimization problem \eqref{eq:Dis Min Prob}, the consensus
constraint, i.e. $x_{i}=x_{j},$ $i,j=1,\ldots,N$, is imposed to
guarantee that the same decision is made by all agents eventually.
In order to find the solution of the problem \eqref{eq:Dis Min Prob},
each agent seeks the minimum of its own objective function, $f_{i}(x_{i})$,
fulfilling its associated inequality constraint $g_{i}(x_{i})\leq0.$
Furthermore, all agents reach consensus on their final states by exchanging
states' information under the graph $\mathcal{{G}}$.

\textcolor{black}{The following assumptions are considered in relation
	to the optimization problem }\eqref{eq:Dis Min Prob}\textcolor{black}{.}
\begin{assumption}
\begin{enumerate} \label{assum:1}

	\item[a.]  The objective functions $f_{i}$, $i=1,\dots,N$ , are strictly
	convex  and twice  continuously differentiable on
	$\mathbb{R^{\text{}}}$. The constraint functions $g_{i}$, $i=1,\dots,N$,
	are convex and twice continuously differentiable on $\mathbb{R^{\text{}}}$.
	\item[b.] The team objective function $\sum_{i=1}^{N}f_{i}(x_{i})$ is radially
	unbounded.
\end{enumerate}
\end{assumption}

\begin{assumption}\label{assum:2}
\textit{(Slater's Condition)} There is some $x^{*}\in\mathbb{R}$
such that $g_{i}(x^{*})\leq0$.
\end{assumption}

\begin{assumption}\label{assum:3}
The graph $\mathcal{{G}}$ is undirected and
has a spanning tree.
\end{assumption}

Intuitively, the problem \eqref{eq:Dis Min Prob} consists of a constrained
convex optimization problem and a consensus problem. The convex constrained
optimization problem can be defined as

\begin{equation}
\begin{array}{c}
\underset{\underset{i=1,\dots,N}{x_{i}}}{min}\sum_{i=1}^{N}f_{i}(x_{i}),\,\\
\text{subject to}\,g_{i}(x_{i})\leq0,\,\forall i\in\mathcal{{N}}.
\end{array}\label{eq:Problem3}
\end{equation}
The consensus problem is
\begin{equation}
\underset{{\scriptstyle t\rightarrow\infty}}{lim}\,(x_{i}-x_{j})=0,\,\,\,\,i,j=1,\ldots,N.\label{eq:consensusProb}
\end{equation}

Based on interior-point method \cite{boyd2004convex}, the convex
optimization problem \eqref{eq:Problem3} can be reformulated as follows,

\begin{alignat}{1}
\begin{array}{c}
\underset{\underset{i=1,\cdots,N}{x_{i}}}{min}\sum_{i=1}^{N}f_{i}(x_{i})-{\displaystyle \frac{\alpha}{\tau}}ln\left(-g_{i}(x_{i})\right).\end{array}\label{eq:1BarrierProb1}
\end{alignat}
where $\tau\in\mathbb{{R}}^{+}$ and $\alpha>1$. The term $-ln\left(-g_{i}(x_{i})\right)$
is referred to as \textit{\textcolor{black}{logarithmic barrier }}\textcolor{black}{function}\textit{\textcolor{black}{.}}
Note that the domain of the logarithmic barrier is the set of strictly
feasible points, i.e. $x_{i}\in\left\{ z\in\mathbb{{R}}:\,g_{i}(z)<0\right\} $.
The logarithmic barrier is a convex function; hence, the new optimization
problem remains to be convex.

Consider the objective function given in \eqref{eq:1BarrierProb1}.
As $x_{i}$ approaches the line $g_{i}(x_{i})=0$, the logarithmic
barrier $-ln\left(-g_{i}(x_{i})\right)$ becomes extremely large.
Thus, it keeps the search dmain within the strictly feasible set.
Note that the initial estimate shall be feasible, i.e. $g_{i}\left(x_{i}(0)\right)<0$,
$i=1,\ldots,N$.
\begin{remark}
	\label{rem:barriermethod}Suppose that the solutions to the optimization
	problem \eqref{eq:Problem3} and \eqref{eq:1BarrierProb1} are $x^{*}$and
	$\widetilde{x}^{*}$, respectively. Then, it can be shown that $f_{i}(x^{*})-f_{i}(\widetilde{x}^{*})=\frac{\alpha}{\tau}$
	\cite{boyd2004convex,wang2011control}. This suggests a very straightforward
	method for obtaining the solution to \eqref{eq:Problem3} with an
	accuracy of $\varepsilon$ by choosing $\tau\geq\frac{\alpha}{\varepsilon}$
	and solving \eqref{eq:1BarrierProb1}. Consequently, as $\tau$ increases,
	the solution to the optimization problem \eqref{eq:1BarrierProb1}
	becomes closer to the solution of \eqref{eq:Problem3} , i.e. as $\tau\rightarrow\infty$,
	$f_{i}(x^{*})-f_{i}(\widetilde{x}^{*})\rightarrow0$ is concluded
	\cite[pp. 568-571]{boyd2004convex}.
\end{remark}

The optimality conditions (so-called centrality conditions) for the convex
optimization problem \eqref{eq:1BarrierProb1} are expressed as \cite{boyd2004convex}

\begin{equation}
\begin{array}{c}
{\displaystyle \sum_{i=1}^{N}\frac{\partial f_{i}(\tilde{x}_{i}^{*})}{\partial x_{i}}}-{\displaystyle \frac{\alpha}{\tau}\frac{{\scriptstyle {\displaystyle {\displaystyle \frac{\partial g_{i}(\tilde{x}_{i}^{*})}{\partial x_{i}}}}}}{g_{i}(\tilde{x}_{i}^{*})}}=0,\\
g_{i}(\tilde{x}_{i}^{*})\leq0.
\end{array}\label{eq:1OptCondition}
\end{equation}
\textcolor{black}{ We now redefine the problem \eqref{eq:1BarrierProb1}
	as
	\begin{equation}
	\underset{\underset{i=1,\cdots,N}{x_{i}}}{min}\sum_{i=1}^{N}f_{i}(x_{i})-{\displaystyle \frac{\alpha}{t+1}}ln\left(-g_{i}(x_{i})\right)\label{eq:1BarrierProb2}
	\end{equation}
	that yields the solution of \eqref{eq:1BarrierProb1} asymptotically.}

\subsection{Centralized Algorithm}

\textcolor{black}{In this subsection, we propose a central paradigm
	to find the solution of the problem }{\eqref{eq:1BarrierProb2}.
}\textcolor{black}{Later, in the next subsection, we realize this
	centralized protocol via a distributed algorithm. }

\textcolor{black}{We utilize the strategy stated in } \eqref{eq:1CentLaw-1} and propose the
following centralized control law to find {\crr the optimal solution for  the optimization problem}
\textcolor{black}{\eqref{eq:1BarrierProb2}}\textcolor{black}{,}

\begin{equation}
\begin{split}
u_{i}(t)&=-\left(\sum_{i=1}^{N}\frac{\partial^{2}L_{i}(x_i,t)}{\partial x_{i}^{2}}\right)^{-1}\left(\sum_{i=1}^{N}\frac{\partial L_{i}(x_i,t)}{\partial x_{i}}\right.\\&\left.+\sum_{i=1}^{N}\frac{\partial^{2}L_{i}(x_i,t)}{\partial x_{i}\partial t}\right)+r_{i}\label{eq:1CentLaw_Collective-1}
\end{split}
\end{equation}
where
\begin{equation}
L_{i}(x_{i},t)=f_{i}(x_{i})-\frac{\alpha}{t+1}ln\left(-g_{i}(x_{i})\right),\label{eq:L}
\end{equation}
and
\begin{equation}
r_{i}=-\beta_{1}\sum_{j\in\mathcal{{N}}_{i}}tanh\beta_{2}(x_{i}-x_{j}),
\end{equation}
in which $t$ represents time and $\beta_{1},\beta_{2}\in\mathbb{R}^{+}$.

Note that the control command \eqref{eq:1CentLaw_Collective-1} consists
of two parts: the first term is to minimize the local objective function,
and the second part is a saturation term associated with the consensus
error.

\begin{definition}
	A network of agents with single-integrator dynamics as \eqref{eq:AgentsDyn}
	is said to reach a \textit{\textcolor{black}{practical consensus
			if $\left|x_{i}(t)-x_{j}(t)\right|\leq\delta_{0}$ }}\textcolor{black}{$\forall i,j\in\mathcal{{N}}$
		for an arbitrarily small $\delta_{0}$.}\end{definition}
	
In the sequel, we will show through the following lemma that the positions
of agents, i.e. $x_{i}$, $i=1,\ldots,N$ , reach the practical  consensus
under the control law \eqref{eq:1CentLaw_Collective-1}.

\begin{lemma}
	\label{prop:consensustheorem}Consider Assumptions \ref{assum:1}.a and \ref{assum:3}. \textcolor{black}{If
		$\left|\omega_{i}-\omega_{j}\right|<\omega_{0}$, $i,j=1,\ldots,N$,
		where $\omega_{i}=\left(\sum_{i=1}^{N}\frac{\partial^{2}L_{i}}{\partial x_{i}^{2}}\right)^{-1}\left(\sum_{i=1}^{N}\frac{\partial L_{i}}{\partial x_{i}}+\sum_{i=1}^{N}\frac{\partial^{2}L_{i}}{\partial x_{i}\partial t}\right)$,
		and $\beta_{1}\sqrt{\lambda_{2}(L)}>\omega_{0}$, then, }there exist
	$t_{1}$ and $\delta_{0}>0$ such that the positions of all  the agents
	with dynamics \eqref{eq:AgentsDyn}\textcolor{red}{{} }\textcolor{black}{under
		the control law \eqref{eq:1CentLaw_Collective-1} satisfy practical
		consesnsus, i.e. $\left|x_{i}(t)-x_{j}(t)\right|\leq\delta_{0}$, $i,j=1,\ldots,N$,
		for $t>t_{1}$.}\end{lemma}
\begin{proof}
	The aggregate dynamics of the agents in \eqref{eq:AgentsDyn} under the control
	law \eqref{eq:1CentLaw_Collective-1} can be written as
	\begin{equation}
	\dot{\bar{x}}=-\beta_{1}Dtanh\left(\beta_{2}D^{\top}\bar{x}\right)+\Omega,\label{eq:ConsProp1}
	\end{equation}
	where $\varOmega=\left[\omega_{1}\ldots\omega_{N}\right]^{\top}$.
	Let the network's consensus error be defined as $\bar{e}_{x}=\Pi\bar{x}$.
	Hence,
	
	\begin{equation}
	\dot{\bar{e}}_{x}=-\beta_{1}Dtanh\left(\beta_{2}D^{\top}\bar{e}_{x}\right)+\Pi\Omega.\label{eq:ConsProp2}
	\end{equation}
	Choose the Lyapanov candidate function
	\begin{equation}
	V(\bar{e}_{x})=\frac{1}{2}\bar{e}_{x}^{\top}\bar{e}_{x}.\label{eq:ConsProp3}
	\end{equation}
	By taking time derivative from $V(\bar{e}_{x})$ along the trajectories
	of $\bar{e}_{x}$, it can be obtained that
	
	\begin{equation}
	\dot{V}(\bar{e}_{x})=-\beta_{1}\bar{e}_{x}^{\top}D\,tanh\left(\beta_{2}D^{\top}\bar{e}_{x}\right)+\bar{e}_{x}^{\top}\Pi\Omega.
	\end{equation}
	Define $\bar{y}=D^{\top}\bar{e}_{x}$, $\bar{y}=[y_{1}\ldots y_{N}]^{\top}$
	. Then, one can say that $-\bar{y}^{\top}tanh(\beta_{2}\bar{y})=\sum_{i}y_{i}tanh(\beta_{2}y_{i})$.
	From the inequality $-\eta tanh(\frac{\eta}{\epsilon})+\left|\eta\right|<0.2785\epsilon$
	for some $\epsilon,\eta\in\mathbb{{R}}$ \cite{polycarpou1993robust},
	it is straightforward to establish  that $-\bar{e}_{x}^{\top}D\,tanh\left(\beta_{2}D^{\top}\bar{e}_{x}\right)<-\left\Vert D^{\top}\bar{e}_{x}\right\Vert _{1}+\frac{N}{\beta_{2}}0.2785$.
	Thus, the following inequalities hold
	
	\begin{eqnarray*}
	\dot{V}(\bar{e}_{x}) & \leq & -\beta_{1}\left\Vert D^{\top}\bar{e}_{x}\right\Vert _{1}+{\scriptstyle \frac{\beta_{1}N}{\beta_{2}}}0.2785+\left\Vert \bar{e}_{x}\right\Vert \left\Vert \Pi\Omega\right\Vert ,\\
	& \leq & -\beta_{1}\left\Vert D^{\top}\bar{e}_{x}\right\Vert +{\scriptstyle \frac{\beta_{1}N}{\beta_{2}}}0.2785+\left\Vert \bar{e}_{x}\right\Vert \left\Vert \Pi\Omega\right\Vert .
	\end{eqnarray*}
	The second inequality arises from the fact that $\left\Vert p\right\Vert \leq\left\Vert p\right\Vert _{1}$
which  holds for any $p\in\mathbb{R}^{n}$. Then, from the assumption
	$\left\Vert \omega_{i}-\omega_{j}\right\Vert <\omega_{0},\,\forall i,j \in \mathcal N $,
	one can attain
	
	\begin{equation}
	\dot{V}(\bar{e}_{x})\leq-\beta_{1}\sqrt{\bar{e}_{x}^{\top}DD^{\top}\bar{e}_{x}}+{\scriptstyle \frac{\beta_{1}N}{\beta_{2}}}0.2785+\left\Vert \bar{e}_{x}\right\Vert \omega_{0}.
	\end{equation}
	According to \textit{Courant-Fischer Formula} \cite{horn2012matrix},
	one can observe that $\bar{e}_{x}^{\top}DD^{\top}\bar{e}_{x}\geq\lambda_{2}(L)\left\Vert \bar{e}_{x}\right\Vert ^{2}$
	, so,
	
	\[
	\dot{V}(\bar{e}_{x})\leq-\beta_{1}\sqrt{\lambda_{2}(L)}\left\Vert \bar{e}_{x}\right\Vert +{\scriptstyle \frac{\beta_{1}N}{\beta_{2}}}0.2785+\left\Vert \bar{e}_{x}\right\Vert \omega_{0}.
	\]
	From the statement of Lemma, we have $\beta_{1}\sqrt{\lambda_{2}(L)}>\omega_{0}$.
	Furthermore, for $\left\Vert \bar{e}_{x}\right\Vert >\frac{\frac{\beta_{1}N}{\beta_{2}}0.2785}{\beta_{1}\sqrt{\lambda_{2}(L)}-\omega_{0}}$,
	we obtain $\dot{V}(\bar{e}_{x})\leq0$. Now, we are ready to invoke
	Lemma \ref{lem:stability} that guarantees that by choosing $\beta_{2}$
	large enough, one can make the consensus error $\delta_{0}$ as small
	as desired.
	
\end{proof}

\begin{remark}
	Assumption $\left|\omega_{i}-\omega_{j}\right|<\omega_{0}$
		in Lemma \ref{prop:consensustheorem} may seem unreasonable
		since, under some mild conditions, it implies boundedness of agents'
		positions, $x_{i}$, $i=1,\ldots,N$. By the following lemma, we will
		prove that the agents' positions stay bounded.\end{remark}
\begin{lemma}
	\label{lem:Boundedness}Consider the dynamics \eqref{eq:AgentsDyn}
	driven by the control command \eqref{eq:1CentLaw_Collective-1}. Then,
	under Assumptions \ref{assum:1}.a and \ref{assum:3}, the solutions of \eqref{eq:AgentsDyn}
	are globally bounded.\end{lemma}
\begin{proof}
	We study boundedness of the solutions of dynamics \eqref{eq:AgentsDyn}
	under the control law \eqref{eq:1CentLaw_Collective-1} via the Lyapunov
	stability analysis. Let us consider the following quadratic Lyapunov function
	\begin{equation}
	W(\bar{x})=\frac{1}{2}(\bar{x}-\bar{x}^{*})^{\top}(\bar{x}-\bar{x}^{*}),\label{eq:LyapFunBounded}
	\end{equation}
	where $\bar{x}^{*}\in\mathbb{{R}}^{n}$ is the optimum point for the
	convex function $\sum_{i=1}^{N}L_{i}(x_{i},t)$. Let us take derivative
	from both sides of \eqref{eq:LyapFunBounded} along the trajectories
	\eqref{eq:AgentsDyn} under {\crb the control law} \eqref{eq:1CentLaw_Collective-1} with
	respect to time. Then, we obtain
	\begin{align}
	\dot{W}(\bar{x}) & =(\bar{x}-\bar{x}^{*})^{\top}\dot{\bar{x}}\nonumber \\
	& =-\sum_{i=1}^{N}\left(x_{i}-x_{i}^{*}\right)\left(\sum_{i=1}^{N}\frac{\partial L_{i}}{\partial x_{i}}+\sum_{i=1}^{N}\frac{\partial^{2}L_{i}}{\partial x_{i}\partial t}\right)\nonumber \\
	& \left(\sum_{i=1}^{N}\frac{\partial^{2}L_{i}}{\partial x_{i}^{2}}\right)^{-1}-\beta_{1}(\bar{x}-\bar{x}^{*})^{\top}Dtanh\beta_{2}D^{\top}\bar{x}\nonumber \\
	& =-\sum_{i=1}^{N}\left(x_{i}-x_{i}^{*}\right)\left({\displaystyle \sum_{i=1}^{N}\frac{\partial f_{i}(x_{i})}{\partial x_{i}}}-{\displaystyle \frac{\alpha t}{\left(t+1\right)^{2}}\sum_{i=1}^{N}\frac{{\scriptstyle {\displaystyle {\displaystyle {\textstyle \frac{\partial g_{i}(x_{i})}{\partial x_{i}}}}}}}{{\textstyle {\textstyle g_{i}(x_{i})}}}}\right)\nonumber \\
	& \left(\sum_{i=1}^{N}\frac{\partial^{2}L_{i}(x_{i},t)}{\partial x_{i}^{2}}\right)^{-1}-\beta_{1}(\bar{x}-\bar{x}^{*})^{\top}Dtanh\beta_{2}D^{\top}\bar{x}.\label{eq:1BoundedLemma2}
	\end{align}
	\textcolor{black}{Define $\hat{L}_{i}(x_{i},t)=f_{i}(x_{i})-\frac{\alpha t}{\left(t+1\right)^{2}}ln\left(-g_{i}(x_{i})\right)$
		. Note that $\hat{L}_{i}(x_{i},t)$ is strictly convex as $\alpha>1$.
		Let the minimizer of $\hat{L}_{i}(x_{i},t)$ be $\hat{x}_{i}^{*}$.
	}\textcolor{black}{One can observe that }\textcolor{black}{{} the minimizers of $\hat{L}_{i}(x_{i},t)$ and $L_{i}(x_{i},t)$,
		$i=1,\ldots,N$, are identical, i.e. $\hat{x}_{i}^{*}=x_{i}^{*}$.
	}\textcolor{black}{On the other hand,}\textcolor{black}{{} due to convexity
		of $\hat{L}_{i}(x_{i},t)$ in $x_{i}$, it holds that $-\left(x_{i}-x_{i}^{*}\right)\frac{\partial\hat{L}_{i}(x_{i},t)}{\partial x_{i}}<\hat{L}_{i}(x_{i}^{*},t)-\hat{L}_{i}(x_{i},t)$,
		$i=1,\ldots,N$. As the inequality $\hat{L}_{i}(x_{i}^{*},t)\leq\hat{L}_{i}(x_{i},t)$
		holds for any $x_{i}$, from the definition of convexity, it can be
		inferred that the first term on the right side of the equality \eqref{eq:1BoundedLemma2}
		is non-positive. }\textcolor{black}{Thus, one obtains}
	
	\begin{eqnarray*}
	\dot{W}(\bar{x}) & \leq & -\beta_{1}(\bar{x}-\bar{x}^{*})^{\top}Dtanh\beta_{2}D^{\top}\bar{x}\label{eq:1BoundedLemma}\\
	& = & -\beta_{1}\bar{x}^{\top}Dtanh\beta_{2}D^{\top}\bar{x}+\beta_{1}\bar{x}^{*\top}Dtanh\beta_{2}D^{\top}\bar{x}\\
	& \leq & -\beta_{1}\left\Vert D^{\top}\bar{x}\right\Vert _{1}+\frac{0.2785\beta_{1}}{\beta_{2}}+\beta_{1}\left\Vert D^{\top}\bar{x}^{*}\right\Vert
	\end{eqnarray*}
	The last {\crb inequality} arises from the inequalities $-\eta tanh(\frac{\eta}{\epsilon})+\left|\eta\right|<0.2785\epsilon$
	\cite{polycarpou1993robust}, with $\epsilon,\eta\in\mathbb{{R}}$,
	and $\left\Vert tanh(\cdot)\right\Vert \leq1$. Furthermore, one can
	easily find $m\in\mathbb{R}$ such that $\left\Vert D^{\top}\bar{x}^{*}\right\Vert \leq m$. This discussion leads to
	\begin{eqnarray*}
	\dot{W}(\bar{x}) & \leq & -\beta_{1}\left\Vert D^{\top}\bar{x}\right\Vert +\frac{0.2785\beta_{1}}{\beta_{2}}+\beta_{1}m\label{eq:boundedEnd}\\
	& = & -\beta_{1}\sqrt{\bar{x}DD^{\top}\bar{x}}+\frac{0.2785\beta_{1}}{\beta_{2}}+\beta_{1}m\\
	& \leq & -\theta\left\Vert \bar{x}\right\Vert +\left(\theta-\beta_{1}\sqrt{\lambda_{2}(DD^{\top})}\right)\left\Vert \bar{x}\right\Vert \\
	&  & +\frac{0.2785\beta_{1}}{\beta_{2}}+\beta_{1}m,\,0<\theta<1\\
	& \leq & -\theta\left\Vert \bar{x}\right\Vert ,\,\,\,\forall\bar{x}\in\mathcal{{B}}.
	\end{eqnarray*}
	where $\mathcal{{B}}=\left\{ \bar{x}\in\mathbb{{R}}^{N}|\left\Vert \bar{x}\right\Vert \geq\frac{\frac{0.2785\beta_{1}}{\beta_{2}}+m\beta_{1}}{\theta-\beta_{1}\sqrt{\lambda_{2}(L)}}\right\} $.
	Now, by Lemma \ref{lem:stability}, it will be certified that $\bar{x}$
	remains bounded.
\end{proof}

\subsection{Distributed Algorithm}

It is obvious that the control law \eqref{eq:1CentLaw_Collective-1}
is not locally implementable since it requires the knowledge of the
whole network as aggregate objective function $\sum_{i=1}^{N}f_{i}(x)$
as well as all inequality constraints $g_{i}(x)\leq0$,\textcolor{black}{{}
	$i=1,\ldots,N$.}  Through the following algorithm, we estimate
\eqref{eq:1CentLaw_Collective-1} in a distributed manner and adopt
it to solve the distributed optimization problem \eqref{eq:Problem3}.

As it follows, each agent generates an internal dynamics to obtain
the estimates of collective objective function's gradients and other
terms which are required for computation of \eqref{eq:1CentLaw_Collective-1}
via only local information in a cooperative fashion. {\crb Consider}  the following
estimator dynamics,

\begin{align}
\dot{\kappa}_{i}(t) & =-c\sum_{j\in\mathcal{{N}}_{i}}sgn\left(\nu{}_{i}(t)-\nu_{j}(t)\right),\label{eq:1EstDynamics1}
\end{align}
where
\begin{equation}
\nu_{i}(t)=\kappa_{i}(t)+\left[{\displaystyle \begin{array}{c}
	\frac{\partial L_{i}(x_{i},t)}{\partial x_{i}}\\
	\frac{\partial^{2}L_{i}(x_{i},t)}{\partial x_{i}\partial t}\\
	\frac{\partial^{2}L_{i}(x_{i},t)}{\partial x_{i}^{2}}
	\end{array}}\right].\label{eq:1IntSignal1}
\end{equation}

From \eqref{eq:1EstDynamics1}, one obtains $\sum_{i=1}^{N}\dot{\kappa}_{i}(t)=0$.
Assume that $\kappa_{i}$, $i=1,\ldots,N$, are initialized such that
$\sum_{i=1}^{N}\kappa(0)=0$. Then, $\sum_{i=1}^{N}\kappa_{i}(t)=0$
is concluded for all $t>0$. Hence, $\sum_{i=1}^{N}\nu_{i}(t)=\sum_{i=1}^{N}\left[{\displaystyle \begin{array}{c}
	\frac{\partial L_{i}(x_{i},t)}{\partial x_{i}}\\
	\frac{\partial^{2}L_{i}(x_{i},t)}{\partial x_{i}\partial t}\\
	\frac{\partial^{2}L_{i}(x_{i},t)}{\partial x_{i}^{2}}
	\end{array}}\right]$. It follows from Theorem 1 in \cite{chen2012distributed} that if
$c>\underset{t}{sup}\left\{ \left\Vert \kappa{}_{i}(x_{i},t)\right\Vert _{\infty}\right\} $,
$\forall i\in\mathcal{{N}}$, then consensus on $\text{\ensuremath{\nu}}_{i}$,
$i=1,\ldots,N$, i.e. {\crb $\left|\nu_{i}(t)-\nu_{j}(t)\right|=0\,\forall i,j \in \mathcal N$,}
is achieved {\crb  over a finite time} {\crb say $T$}.
With $\nu_{i}(t)=\nu_{j}(t)$, the following holds,
\begin{equation}
\nu_{i}(t)=\frac{1}{N}\sum_{i=1}^{N}\left[{\displaystyle \begin{array}{c}
	\nu_{i1}\\
	\nu_{i2}\\
	\nu_{i3}
	\end{array}}\right].\label{eq:1EstDyn4}
\end{equation}
\textcolor{black}{where $\nu_{i1}=\frac{\partial L_{i}(x_{i},t)}{\partial x_{i}},$
	$\nu_{i2}=\frac{\partial^{2}L_{i}(x_{i},t)}{\partial x_{i}\partial t}$,
	and $\nu_{i3}=\frac{\partial^{2}L_{i}(x_{i},t)}{\partial x_{i}^{2}}$.
}
\begin{theorem}
	Suppose that Assumptions \ref{assum:1}, \ref{assum:2}, and \ref{assum:3} hold. Moreover, $\sum_{i=1}^{N}\kappa_{i}(0)=0$
	and $c>\underset{t}{sup}\left\{ \left\Vert \kappa{}_{i}(x_{i},t)\right\Vert _{\infty}\right\} ,\,\forall i\in\mathcal{{N}}$.
	Then, the protocol
	\begin{equation}
	u_{i}(t)=-\nu_{i3}^{-1}\left(\nu_{i1}+\nu_{i2}\right)+r_{i},\,i=1,\ldots,N\label{eq:1DisEstLaw}
	\end{equation}
	will drive the agents \eqref{eq:AgentsDyn} to the solution of the
	distributed convex optimization problem \eqref{eq:Problem3}.\end{theorem}
\begin{proof}
	{\crb Let us} define the following Lyapunov candidate function $V=\frac{1}{2}\left(\sum_{i=1}^{N}\nu_{i1}\right)^{2}.$
	After calculating time derivate of $V$, the following holds,
	\begin{align}
	\dot{V} & =\left(\sum_{i=1}^{N}\nu_{i1}\right)\left(\sum_{i=1}^{N}\nu_{i3}u_{i}+\nu_{i2}\right).
	\end{align}
	From \eqref{eq:1DisEstLaw} , we have
	\begin{eqnarray}
	\dot{V} & = & -\left(\sum_{i=1}^{N}\nu_{i1}\right)^{2},
	\end{eqnarray}
	in which we used the equalities $\nu_{i3}=\nu_{j3},\,\forall i,j\in\mathcal{{N}}$
	for $t>T$, and $\sum_{i=1}^{N}r_{i}=0$.  Hence, $\dot{V}\leq0,\,\forall t>T$.
	On the other hand, we assert that $x_{i}$, $i=1,\ldots,N$, stay
	bounded after a finite time as the agents' dynamics are locally Lipschitz
	and their inputs are bounded. This means that for $t\leq T$, we have
	$x_{i}\in\mathbb{{R}},\forall i \in \mathcal N $. Now, we can do stability analysis
	from $T$ onwards.

From the above inequality, it follows that ${\displaystyle \sum_{i=1}^{N}\nu_{i1}}$
remains bounded in $\mathbb{{R}^{\text{n}}\bigcup}\{\infty\}$, i.e.
it belongs to $\mathcal{{L}^{\infty}}$ space. With integrating from
both sides of equality \eqref{eq:Lemma1_1}, in the view of passivity
of $V(x,t)$, we have
\begin{gather*}
\int_{0}^{R}\left(\sum_{i=1}^{N}\nu_{i1}\right)^{2}dt=-\int_{0}^{R}\dot{V}dt\\
=-V(R)+V(0)\leq V(0).
\end{gather*}
Therefore, $\sum_{i=1}^{N}\nu_{i1}\in\mathcal{{L}}^{2}$.  By means
of Barbalat's lemma \cite{tao1997simple}, we have $\sum_{i=1}^{N}\nu_{i1}=0$
as $t\rightarrow\infty$. Thereby, the first optimality condition
in \eqref{eq:1OptCondition} is asymptotically satisfied.

\textcolor{black}{We now illustrate  that the second optimality condition
	in \eqref{eq:1OptCondition} also holds. Suppose that $g_{i}\left(x_{i}(0)\right)<0$
	$\forall i$. We  do the proof by contradiction to establish
	that $g_{i}\left(x_{i}(t)\right)<0$ for $t>0$ holds. Assume that we had
	$g_{i}\left(x_{i}(t_{1}^{-})\right)<0$ and $g_{i}\left(x_{i}(t_{1}^{+})\right)>0$
	for some $i$ and a finite $t_{1}>0$. Due to continuity of the function
	$g_{i}\left(\cdot\right)$, $g_{i}(x_{i}(t_{1}))$ would be zero.
	This implies that $\sum_{i=1}^{N}\nu_{i1}$ becomes unbounded at $t_{1}$
	that contradicts the fact that $\sum_{i=1}^{N}\nu_{i1}\in\mathcal{{L}^{\infty}}$.
	Hence, the inequality $g_{i}(x_{i}(t))<0$ with $g_{i}\left(x_{i}(0)\right)<0$
	holds for $t>0$. This ends the proof. }
\end{proof}

\subsection{Numerical Example}

This section presents simulation studies using Matlab/ Simulink software
for a network of four agents with dynamics according to  \eqref{eq:1SingIntDym-1}
driven by the proposed distributed algorithm \eqref{eq:1DisEstLaw}.
We consider the constrained convex optimization problem \eqref{eq:min Prob}
with obejctive functions $f_{1}(x)=(x+2)^{2}$, $f_{2}(x)=x^{2}$,
$f_{3}(x)=(x-10)^{2}$, and $f_{4}(x)=(x-2)^{2}$and constraints $g_{1}(x)=x-1$,
$g_{2}(x)=x-2$, and $g_{3}(x)=x-4$. Note that these functions are
all smooth and convex. In our simulation, the information sharing
graph $\mathcal{{G}}$ is set as: $1\Leftrightarrow2\Leftrightarrow3\Leftrightarrow4$.
The initial conditions are set as $x_{0}=\left[\begin{array}{cccc}
0 & 1 & 3 & 0\end{array}\right]^\top.$ The evolution of agents' trajectories is depicted in the Fig 1. As
shown, all agents meet each other and, then, converge towards the
optimal point of the collective objective function which is one in this case.

\section{Conclusions}\label{sec:conclude}

We investigated the problem of \textcolor{black}{distributed optimization
	f}or undirected networks of single-integrator agents. Here, agents
shall reach an agreed point that minimizes a collective convex objective
function with respect to local inequality constraints. A centralized
control law, which yields the optimal solution to this problem, and   consists of\textcolor{black}{{} a saturation consensus
	part and an optimization part based on the interior-point method was
	proposed}. \textcolor{black}{To illustrate the convergence of the
	proposed algorithm, we first established that the proposed consensus
	protocol provides practical consensus, i.e. all agents will have the
	same decision eventually, perhaps with a small admitted error. We then
	suggested a distributed estimator as a tool to estimate some terms
	within the protocol associated with the global knowledge, which is only partially available to agents with the network.
	It was proved that the presented distributed algorithm converges to
	the solution of the original constrained convex optimization problem. }Finally,
to evaluate the performance of our work, a numerical example was presented.

\begin{figure}[tbh]
	\includegraphics[scale=0.12]{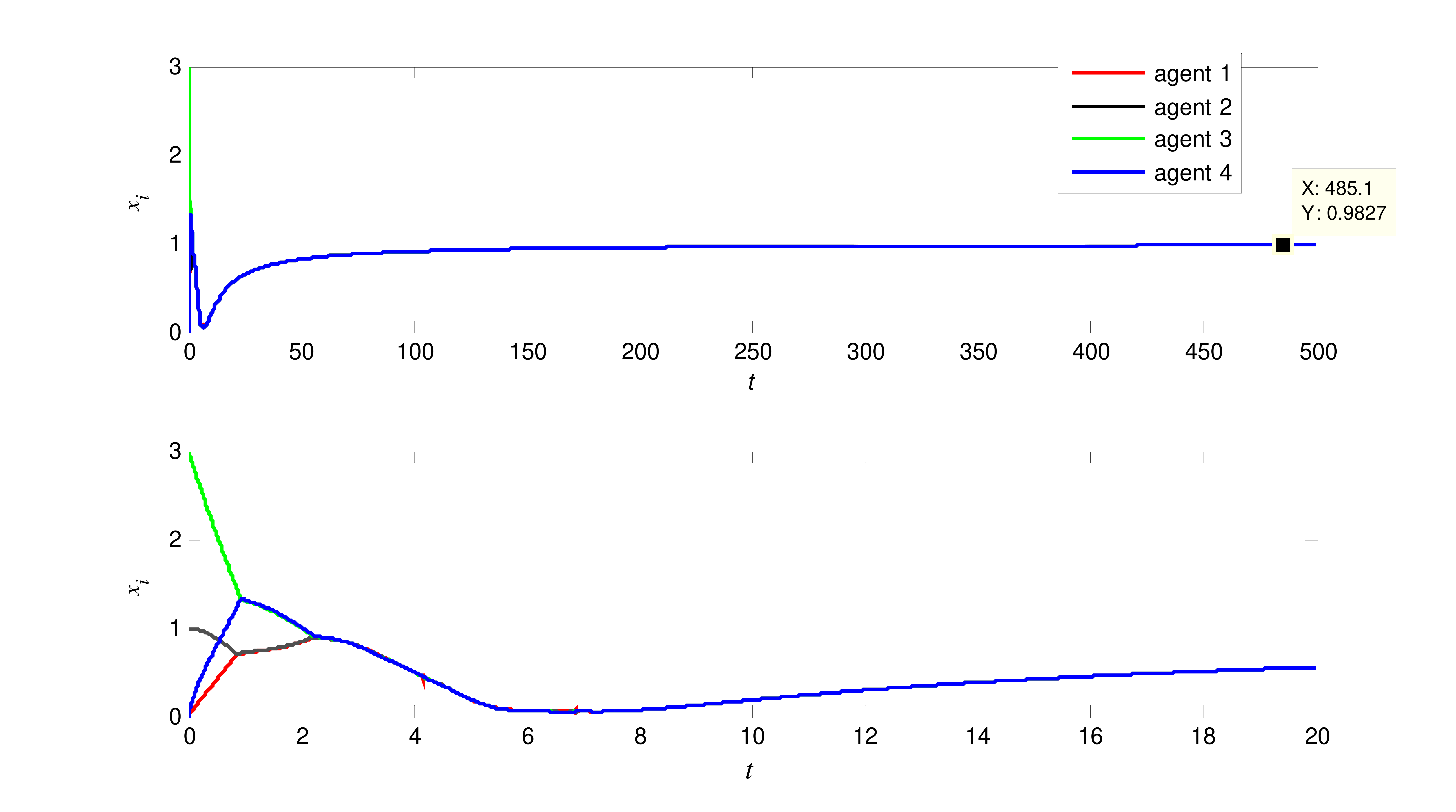}
	
	\caption{Trajectories of single-integrator agents towards the optimal point
		under the protocol \eqref{eq:1DisEstLaw}.}
\end{figure}

\bibliographystyle{plain}
\bibliography{paper2}
\end{document}